\documentclass[conference]{IEEEtran}
\IEEEoverridecommandlockouts
\usepackage{cite}
\usepackage{amsmath,amssymb,amsfonts}
\usepackage{graphicx}
\usepackage{textcomp}
\usepackage{xcolor}
\def\BibTeX{{\rm B\kern-.05em{\sc i\kern-.025em b}\kern-.08em
    T\kern-.1667em\lower.7ex\hbox{E}\kern-.125emX}}
\usepackage{floatpag}
\usepackage{float}    
\usepackage[framemethod=tikz]{mdframed}
\usepackage{bm}

\usepackage{breqn}
\usepackage{hyperref}
\usepackage{mathtools}

\usepackage{algorithm}
\usepackage{algpseudocode}

\algnewcommand\algorithmicforeach{\textbf{for each}}
\algdef{S}[FOR]{ForEach}[1]{\algorithmicforeach\ #1\ \algorithmicdo}

\usepackage{lipsum}

\usepackage{amsthm}
\newtheorem{theorem}{Theorem}
\newtheorem{corollary}{Corollary}
\newtheorem{lemma}{Lemma}
\newtheorem{definition}{Definition}
\newtheorem{assumption}{Assumption}

\DeclareMathOperator*{\argmin}{arg\,min}
\usepackage{subcaption}
\newcommand{\squeezeup}{\vspace{-2.5mm}}
\usepackage{todonotes}

\begin{document}

\title{A Probabilistic Guidance Approach to Swarm-to-Swarm Engagement Problem}


\author{%
  Samet Uzun%
  \rlap{\textsuperscript{1}},
  Nazım Kemal Üre %
  \rlap{\textsuperscript{1}}
}

\maketitle

\footnotetext[1]{Istanbul Technical University, Department of Aeronautics and Astronautics, Istanbul, 34469, Turkey \{uzunsame, ure\}@itu.edu.tr}

\begin{abstract}
This paper introduces a probabilistic guidance approach for the swarm-to-swarm engagement problem. The idea is based on driving the controlled swarm towards an adversary swarm, where the adversary swarm aims to converge to a stationary distribution that corresponds to a defended base location. The probabilistic approach is based on designing a Markov chain for the distribution of the swarm to converge a stationary distribution. This approach is decentralized, so each agent can propagate its position independently of other agents. Our main contribution is the formulation of the swarm-to-swarm engagement as an optimization problem where the population of each swarm decays with each engagement and determining a desired distribution for the controlled swarm to converge time-varying distribution and eliminate agents of the adversary swarm until adversary swarm enters the defended base location. We demonstrate the validity of proposed approach on several swarm engagement scenarios.
\end{abstract}


\todototoc

\section{Introduction} \label{sec:intro}

In nature, some challenging tasks can be completed more robustly and efficiently by using a swarm with a large number of small agents rather than using a few large ones. Animal communities such as bees or ants can aggregate together and exhibit collective behaviors for gathering food or avoiding a threat. Hence, they can improve their long term survival chances. These natural display of swarm behaviours is attempted to be applied to engineered multi-agent systems by many scientists. Efficient swarm guidance algorithms are one of the most important parts to be designed to implement such systems.

\begin{figure}[!hbt]
    \centering
    \includegraphics[width=0.4\textwidth]{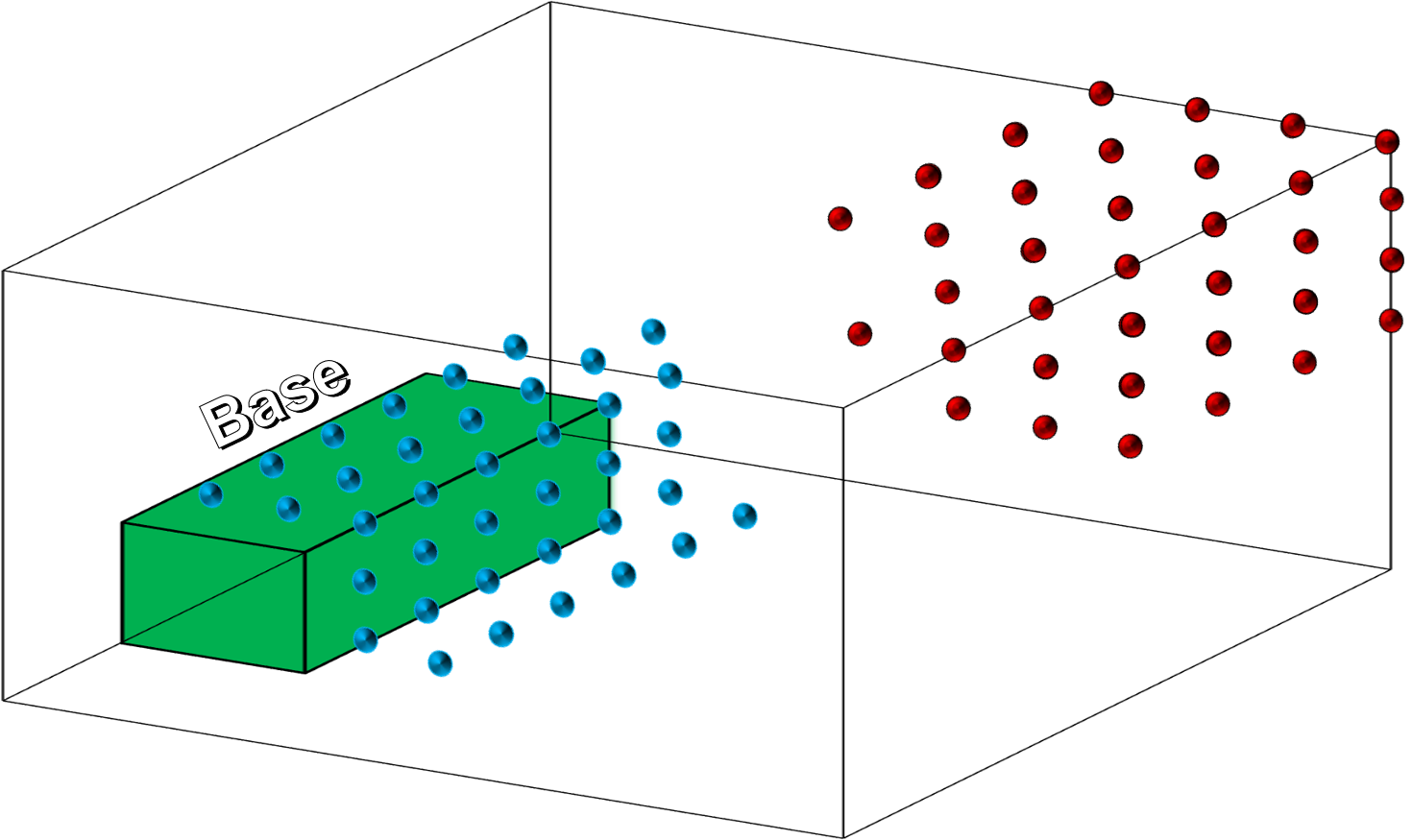}
    \caption{Representation of the base location, controlled (blue) swarm, and adversary (red) swarm. While the red swarm tries to converge to the stationary distribution of the base location, the blue swarm tries to engage with the time-varying distribution of the red swarm.}
    \label{fig:S2S}
\end{figure}

This paper introduces a probabilistic guidance approach for the swarm-to-swarm engagement problem. To the best of our knowledge, this work is the first approach to swarm-to-swarm engagement problem. The idea is based on driving a controlled swarm towards an adversary swarm to eliminate agents of the adversary swarm, where the adversary swarm aims to converge to a stationary distribution that corresponds to a defended base location. Proposed probabilistic swarm guidance approaches in the earlier literature work under the assumption that a swarm aims to converge to a stationary distribution. Hence, these methods can not be directly applied to the swarm-to-swarm engagement problem, where the controlled swarm should converge to the time-varying distribution of the adversary swarm to eliminate agents of the adversary swarm. The situation is illustrated in Figure \ref{fig:S2S}, where there are a controlled (blue) swarm, an adversarial (red) swarm, and a stationary distribution, which corresponds to the defended base location. While the red swarm acts according to their pre-determined Markov chain to converge the stationary distribution of the base location, blue swarm designs a Markov chain to engage and eliminate agents of the red swarm. In this paper, we propose a method that determines a stationary distribution for the blue swarm to engage with the red swarm. The method is based on to converge projection of the distribution of the red swarm respect to the Markov chain of the red swarm. The number of agents of the red swarm that enter the defended base location can be bounded in this method. The strategy of the blue swarm is to eliminate desired number of agents of the red swarm furthest from the defended base location. 

A scheme for the swarm-to-swarm engagement is given in Figure \ref{fig:Thesis_Schema_1}. Red swarm propagates with the probabilistic guidance algorithm using their pre-determined Markov chain to converge defended base location. Blue swarm determines the stationary distribution to engage with the red swarm and synthesize a Markov matrix to converge to the determined their distribution. Then, the blue swarm also propagates with the probabilistic guidance algorithm. After both swarms are propagated, agents are eliminated with respect to engagement dynamics. The same processes continue with the new distributions of the blue and red swarms.

\begin{figure}[!hbt]
\centering
\includegraphics[width=0.5\textwidth]{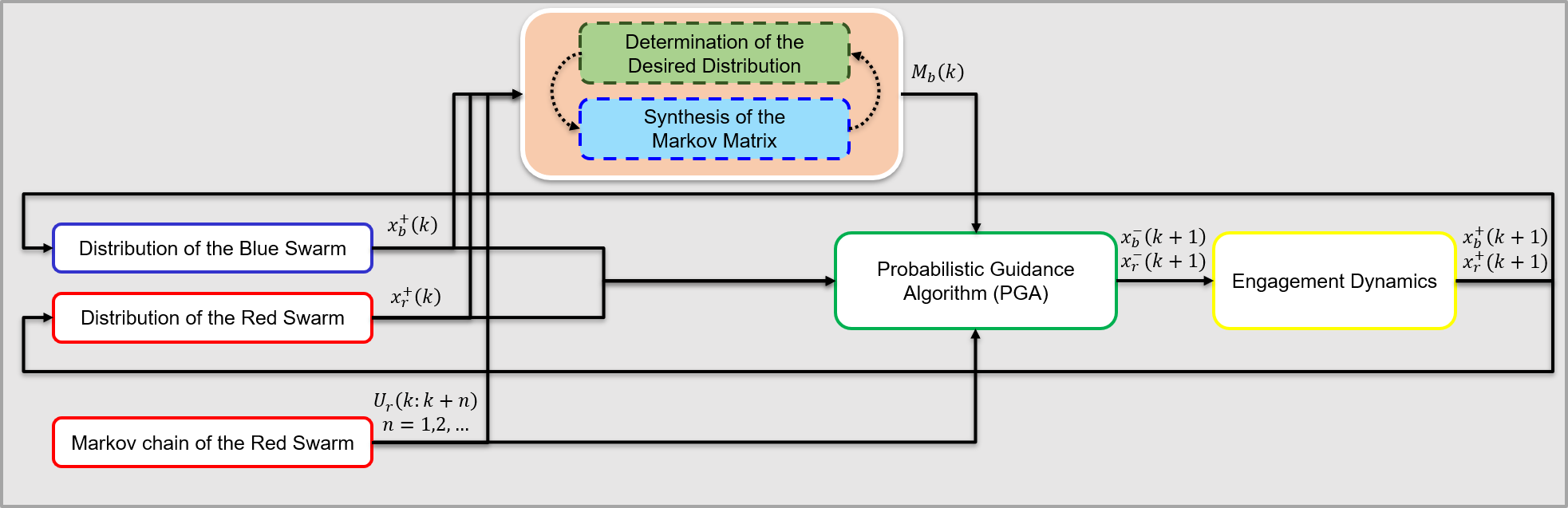}
\caption{Swarm-to-swarm engagement scheme.}
\label{fig:Thesis_Schema_1}
\end{figure}

Swarm engagement has the potential to be one of the key aspects of the future military scene. For instance, DARPA's OFFSET program \cite{DARPA} focuses on the development of rapid generation of tactics for offensive swarms. Since the conventional missile defense systems are not cost-effective for defending against such adversary swarm attacks, we believe that employing a defensive swarm and utilizing swarm-to-swarm engagement algorithms can provide a much more effective solution against such attacks.

\subsection{Related Works} \label{sec:related_work}
In earlier literature, many deterministic methods are developed for the guidance of swarms containing up to $10-20$ agents. However, these methods are computationally infeasible for the systems that contain agents ranging from hundreds to tens of thousands \cite{lumelsky1997decentralized, richards2002spacecraft, tillerson2002co, scharf2003survey, kim2004multiple, ramirez2010distributed}. In order to handle this scalability issue, density based deterministic and probabilistic swarm guidance methods are developed for both discrete and continuous state spaces.

For continuous state spaces, several probabilistic and deterministic swarm guidance methods are developed. Density distribution of the swarm agents is probabilistically controlled to converge a pre-determined distribution by a diffusing swarm of robots that take local measurements of an underlying scalar field in \cite{elamvazhuthi2016coverage}. In \cite{zhao2011density}, an example of deterministic density control that is inspired by Smoothed Particle Hydrodynamic (SPH) is illustrated for group motion and segregation. A deterministic velocity field method is presented to drive the swarm of robots smoothly to the desired density distribution in \cite{eren2017velocity}.

There are also probabilistic methods that consider the swarm as a statistical ensemble and treat the guidance problem as convergence to a desired swarm density distribution in a discrete state space.

In \cite{accikmecse2012markov}, Metropolis-Hastings algorithm is used to design a Markov chain for guidance of swarm agents to the desired swarm density distribution in discrete state space. 

Modeling objective function and constraints as linear matrix inequalities (LMI), convex optimization techniques are used to design the corresponding Markov chain in \cite{accikmecse2015markov}. Then, in order to increase the convergence rate and satisfy some safety constraints, convex optimization approach is extended using density feedback in \cite{demir2014density} and \cite{demir2015decentralized}, respectively. On the other hand, none of these methods considers the number of transitions of the agents. In order to minimize the number of transitions of the agents, swarm guidance is formulated as an optimal transport problem in \cite{bandyopadhyay2014probabilistic}. However, the computation time of the algorithm increases rapidly with increasing the dimension of the state space and performance of the algorithm drops dramatically with estimation errors of the current swarm distribution. An efficient time-inhomogeneous Markov chain approach is proposed for probabilistic swarm guidance (PSG-IMC algorithm) to minimize the number of transitions of the agents in \cite{bandyopadhyay2017probabilistic}. 

However, communication between all bins has to be established for all these mentioned feedback based methods to handle global current distribution for feedback. Thus, this method is modified to work with local information in \cite{jang2018local}. In this method, a local-information based probabilistic guidance method is designed on the time-inhomogeneous Markov chain approach proposed in \cite{bandyopadhyay2017probabilistic}. However, in both global and local-information based time-inhomogeneous Markov chain approaches, convergence rate of the swarm distribution decreases sharply as transition constraints of the agents increases.

In order to converge to the desired steady-state distribution at high convergence rate respect to earlier methods via minimal number of state transitions to minimize resource usage, decentralized state-dependent Markov chain synthesis method is proposed in \cite{uzun2021-1}. In mentioned earlier methods \cite{accikmecse2012markov, accikmecse2015markov, bandyopadhyay2017probabilistic, jang2018local, uzun2021-1}, a shortest path algorithm is used for agents to propagate them from zero density states of desired distribution, which is named as transient states, to non-zero density states of desired distribution which is named as recurrent states. In \cite{uzun2021-2}, a density diffusion based method is proposed for transient states of the desired distribution to increase the convergence rate of the swarm distribution. Also, in all these methods \cite{accikmecse2012markov, accikmecse2015markov, bandyopadhyay2017probabilistic, jang2018local, uzun2021-1, uzun2021-2}, it is assumed that recurrent states of the desired distribution are strongly connected among themselves which means transition between all recurrent states without using transient states of the desired distribution is possible. In \cite{uzun2021-3}, a method is proposed to converge to the desired distributions with disconnected parts. This method allows multiple Markov chain to be synthesized for the each strongly connected part of the recurrent states of the desired distribution and to combine these Markov chains to converge desired distribution with disconnected parts.


\subsection{Main Contributions}
In this work, we develop a problem formulation and propose an efficient algorithm to solve the swarm-to-swarm engagement problem. We introduce the objectives and engagement dynamics of the controlled and adversary swarms. Then, we propose a method that deal with the time-varying distribution of the adversary swarm. The method is based on to converge projection of the red swarm on any determined boundary bins. We proved that the expected number of red agents that enter the defended base location can be bounded under some mild assumptions. We propose an algorithm to maximize the distance between the defended base location and boundary bins as keeping the ratio of red agents that enter the defended base location under a pre-determined ratio.

\subsection{Organization}
The paper is organized as follows. Section \ref{sec:background} introduces the probabilistic guidance problem formulation and 
our Markov chain synthesis method.
Section \ref{sec:ssep} introduces the objectives and engagement dynamics of swarms and presents an algorithm to eliminate desired number of agents of the adversary swarm until adversary swarm enters the defended base location. Section \ref{sec:results} presents numerical results across different swarm engagement scenarios.

\textit{Notation:}
In this paper, the time index is donated by a right subscript and the agent index are donated as lower-case right superscript. $\bm{0}$ and $\bm{1}$ are zero matrix and matrix of ones in appropriate dimensions. $V \setminus W$ is the elements in set $V$ that are not in set $W$. $\mathcal{P}$ denotes probability of a random variable. $M>(\geq)0$ implies that $M$ is a positive (non-negative) matrix. $e_i$ is a vector in appropriate dimensions with its $i$-th element is $+1$ and its other entries are zeros. $\sigma (A)$ is the set of eigenvalues $\lambda$ of $A$, $\rho(A)$ is the spectral radius of $A$ ($ \max_{\lambda \in \sigma(A)} |A|$). $\odot$ represents the Hadamard (Schur) product. Upper-case right superscript '$-$' and '$+$' imply the prior and posterior form of a scalar or vector for an event. $\left\| x \right\|_1 = \sum_{i} |x[i]|$ denotes the $L_1$ norm of vector $x$. $(v_1, v_2, ... , v_n)$ represents a vector, such that $(v_1, v_2, ... , v_n) \equiv [v_1^T, v_2^T, ... , v_n^T]^T$ where $v_i$ have arbitrary dimensions. $diag(A) = (A[1,1], A[2,2], ... , A[n,n])$ for $A \in R^{nxn}$. 

\section{Background} \label{sec:background}

Most of the underlying definitions and baseline algorithms are based on \cite{accikmecse2012markov}, in this section, we briefly review this material for completeness.

\subsection{Swarm Distribution Guidance Problem} \label{sec:sdgp}

\begin{definition}
\textit{(Bins)} The operational region, which the swarm agents are distributed, is denoted as $\mathcal{R}$. The region is assumed to be partitioned as the union of $m$ disjoint regions, which are referred to as bins $R_i$, $i = 1,...,m$, such that  $\mathcal{R} = \cup_{i=1}^{m} R_i$, and $R_i \cap R_j = \varnothing$ for $i \neq j$.
\end{definition}

\begin{definition}
\textit{(Density distribution of the swarm)} Let an agent have position $r(k)$ at time index $k \in Z^+$. Let $x(k)$ be a vector of probabilities, $\bm{1}^T x(k) = 1$, such that the $i$-th element $x[i](k)$ is the probability of the event that this agent will be in bin $R_i$ at time $k$. Consider a swarm comprised of $N$ agents. Each agent is assumed to act independently of the other agents, so that the following equation holds for $N$ separate events,
\begin{equation} \label{eq:x(k)}
    x[i](k) := \mathcal{P}(r_l(k) \in R_i), \quad l = 1,...,N,\\
\end{equation} 
where $r_l(k)$ denotes the position of the $l$’th agent at time index $k$, and the probabilities of these $N$ events are jointly statistically independent. We refer to $x(k)$ as the density distribution of the swarm.
\end{definition}

Let density distribution of a swarm at time $k$ is donated as $x(k)$ and the swarm is comprised of $N$ agents. Distribution of the number of agents of the swarm at time $k$ is donated as $s(k)$ and defined with the following equation,
\begin{equation}
    s(k) = N x(k).
\end{equation}

\begin{definition} \label{def:lim_to_v}
\textit{(Desired steady-state distribution)} It is desired to guide the agents to a specified steady-state distribution described by a probability vector $v \in \mathbb{R}^m$,
\begin{equation} \label{eq:lim_to_v}
    \lim_{k\to\infty} x(k) = v.\\
\end{equation} 
\end{definition}

States of the desired distribution can be classified as in the following definition.
\begin{definition} \label{def:recur_trans}
\textit{(Recurrent and transient states)} The states that have non-zero elements in the desired distribution $v$ are called recurrent states. The other states with zero elements in the desired distribution $v$ are called transient states.
\end{definition}

The main idea of the probabilistic guidance is to drive the propagation of probability vector $x$, instead of individual agent positions $\{r_l(k)\}_{l=1}^N$. Swarms are formed as a statistical ensemble of agents to facilitate the guidance of swarm problem. Although the distribution of swarm agent positions $n/N$ is usually different from $x$ numerically, it is equal to $x$ on the average. Using the law of large numbers, $x$ can be made arbitrarily close to $n/N$ as the number of agents increases.

\subsection{Decentralized Probabilistic Swarm Guidance} \label{sec:pga}
\subsubsection{Probabilistic Guidance Algorithm} 

\begin{assumption} \label{asm:agents_capability}
\textit{(Agent's capability)} All agents can determine their current bins to use their stochastic policy for the transition.
\end{assumption}

\begin{definition}
\textit{(Stochastic policy)}
All swarm agents is propagated at time $k$ with a column stochastic matrix $M(k) \in \mathbb{R}^{m \times m}$ that is called as Markov matrix \cite{horn2012matrix}. Then, $M(k)$ has to satisfy
\begin{equation} \label{eq:desiredss}
    \bm{1}^T M(k) = \bm{1}^T, \quad M(k) \geq \bm{0} \\. 
\end{equation}

The entries of matrix $M(k)$ are defined as transition probabilities. Specifically, for any $k \in N_+$ and $i, j = 1, . . . , m$,
\begin{equation}
    M[i,j](k) = e_i^T M(k) e_j = \mathcal{P}(r(k+1) \in R_i | r(k) \in R_j ). \\
\end{equation} 
\textit{i.e.}, 
an agent in bin $j$ transitions to bin $i$ with probability $M[i, j](k)$. 
\end{definition}

The constraints $M(k) \geq \bm{0}$ and $\bm{1}^T M(k) = \bm{1}^T$ simply implies that the probability of moving from one bin to another bin is nonnegative and the sum of probabilities of transition from a given bin to another bin is equal to $1.0$. The Markov matrix is supplied each of the agents to propagate their position, which is only depending on their own states.

The matrix $M(k)$ determines the time evolution of the probability vector $x$ as
\begin{align} \label{eq:xt1mxt}
\begin{split}
    x(k+1) = M(k)x(k), \quad k = 0,1,2,... \\
    \text{\: with \:} x(0) \geq \bm{0} \text{\: and \:} \bm{1}^Tx(0) = 1.
\end{split}
\end{align}

Time evolution of the probability vector $x$ from any time $k$ to any time $k+n$ can be written with forward matrix product as,
\begin{align}
    x(k + n) &= M(k + n - 1) ... M(k + 1) M(k) x(k)\\
    x(k + n) &= U(k:k+n) x(k).
\end{align}

\begin{algorithm}[H]
    \begin{enumerate}
      \item Each agent determines its current bin, e.g., $r_l(k) \in R_i$.
      \item Each agent generates a random number $z$ that is uniformly distributed in [0, 1].
      \item Each agent transitions to bin $j$, \textit{i.e.}, $r_l(k + 1) \in R_j$, if
        $
      \begin{cases}
          $$ \sum_{s=1}^{j-1} M[s,i](k) \leq z \leq \sum_{s=1}^{j} M[s,i](k)$$.   \\
      \end{cases}
        $
    \end{enumerate}
    \caption{Probabilistic Guidance Algorithm}
\label{alg:pga}
\end{algorithm}
The probabilistic guidance algorithm is given in the Algorithm \ref{alg:pga}. The first step of the algorithm is to determine the agent’s current bin. In the following steps, each agent samples from a discrete distribution and transitions to another bin depending on the column of the Markov matrix, which is determined by the agent’s current bin.

\subsubsection{Convergence to Desired Steady-State Distribution} \label{sec:cdssd}

Assume that, desired distribution of the swarm is donated by the vector $v$. The main idea of the probabilistic guidance law is to synthesize a Markov matrix that satisfies the condition given in the Definition \ref{def:lim_to_v}. 


\subsubsection{Transition Constraints}
The transition between two bins is represented by an edge of a directed graph, where the adjacency matrix of the directed graph is defined similar to the second section of \cite{mesbahi2010graph}.
\begin{definition} \label{def:motcons}
\textit{(Transition Constraints)} Adjacency matrix is used to restrict the allowable transitions of the agents. $A_a[i,j] = 1$ if the transition from bin $i$ to bin $j$ is allowable, and is zero otherwise. Transition constraint of the agents are restricted with the following inequality:
\begin{equation} \label{eq:motcons}
    (\bm{1} \bm{1}^T - A_a^T) \odot M = \bm{0}.\\
\end{equation}
\end{definition}

\begin{assumption} \label{asm:strongly_connected}
    (\textit{Strongly Connected}) It is assumed that all bins is strongly connected by the adjacency matrix which means there exists a directed path between all bins, or equivalently, $(I + A_{a})^{m-1} > \bm{0}$ for $A_{a} \in R^{m \times m}$ \cite[section 6.2.19]{horn2012matrix}. Then, there exist a directed path between all bins.
\end{assumption}

\begin{definition}
    (\textit{Neighbor Bins}) All bins $j$ that satisfy the condition $A_a[j,i] = 1$ are neighbor bins of the bin $i$.
\end{definition}

\subsection{Synthesis of the Markov Matrix} \label{sec:somm}
Because of the convergence rate performance via minimal number of transitions, decentralized state-dependent Markov chain synthesis method proposed in \cite{uzun2021-1} is used to synthesize the part of the Markov matrices related to the recurrent states. The other part of the Markov matrices related to the transient states is synthesized using density diffusion based method proposed in \cite{uzun2021-2} again due to the high convergence rate. In addition, the strategy proposed in \cite{uzun2021-3} is used for the case that the desired distribution have disconnected parts.

\section{Swarm-to-swarm Engagement Problem Formulation and Proposed Engagement Solution}\label{sec:ssep}

\subsection{Introducing the Objectives and Dynamics of the Swarms}
\begin{definition}
\textit{(Blue and red swarms)} There are two types of swarms in the swarm-to-swarm engagement problem, which are controlled (blue) and adversarial (red) swarms. Density distributions, Markov matrices, forward matrix products of Markov matrices, distribution of number of agents and total number of agents for the blue and red swarms are donated as right subscripts '$b$' and '$r$' (i.e. $x_b$, $x_r$, $M_b$, $M_r$, $U_b$, $U_r$, $s_b$, $s_r$, $N_b$ and $N_r$.)
\end{definition}

\begin{definition}
\textit{(Elimination dynamics)} The agents, which are in the same bins and belong to different swarms, eliminate each other in equal numbers. Only the remaining agents that are included in the swarm with the higher number of agents survive in these bins. The number of agents and swarm distributions are donated with a superscript '$-$' before the elimination process and donated with a superscript '$+$' after the elimination process. During the elimination process, distribution of the number of blue and red swarms is changing with the following equations,
\begin{equation} \label{eq:xbre_1}
\begin{aligned}
    s_b^+(k) = \max \bigg( \Big( s_b^-&(k) - s_r^-(k) \Big), \bm{0} \bigg),\\
    s_r^+(k) = \max \bigg( \Big( s_r^-&(k) - s_b^-(k) \Big), \bm{0} \bigg).\\
\end{aligned}
\end{equation}
\end{definition}

Hence, the total number of agents and distributions of the swarms are reconfigured as in the following equations,
\begin{equation} \label{eq:xbre_2}
\begin{aligned}
    N_b^+(k) =& \bm{1}^T s_b^+(k),\\
    N_r^+(k) =& \bm{1}^T s_r^+(k),\\
    x_b^+(k) =& \frac{s_b^+(k)}{N_b^+(k)},\\
    x_r^+(k) =& \frac{s_r^+(k)}{N_r^+(k)}.\\
\end{aligned}
\end{equation}

\begin{definition}
\textit{(Transition dynamics of the blue and red swarms)} Each swarm synthesizes a Markov matrix using the method, which is discussed in Section \ref{sec:somm}, to converge their specified distributions. The time evaluations of the density distributions of the swarm with these Markov matrices are given as 
\begin{equation}  \label{eq:timeeval}
\begin{aligned}
    x_b^-(k+1) =& M_b(k) x_b^+(k),\\
    x_r^-(k+1) =& M_r(k) x_r^+(k).\\
\end{aligned}
\end{equation}
\end{definition}

\begin{definition} \label{def:lim_4_b}
\textit{(Desired distribution for the red swarm)} The purpose of the red swarm is to converge to the desired steady-state distribution, which corresponds to the base location and denoted as $v_b$,
The desired change of the distributions of the red swarm for any time-step $k$ are formulated as
\begin{equation} \label{eq:br2rb}
\begin{aligned}
    \lim_{n\to\infty} U_r(k:k+n)x_r^+(k) =& v_b.\\
\end{aligned}
\end{equation}
\end{definition}

\subsection{Determination of the Stationary Distribution for the Blue Swarm} \label{sec:doct}

The purpose of the red swarm is to converge to the base location as given in the Eq. (\ref{eq:br2rb}). The purpose of the blue swarm is to converge to the time-varying distribution and eliminate desired number of agents of the red swarm while the red swarm furthest from the defended base location. Hence, blue swarm has to know the distribution of the red swarm and should have knowledge about the Markov chain of the red swarm to create a strategy to engage with the red swarm. These knowledge can be extended to estimate the Markov chain of the red swarm for future studies.

\begin{assumption}
The Markov chain and the current distribution of the red swarm are available to blue swarm.
\end{assumption}

The strategy is to select boundary bins and converge to the projection of the distribution of red swarm on these boundary bins with respect to the Markov chain of the red swarm. The number of agents of the red swarm that enter the base location can be bounded with this strategy which is given in the Theorem \ref{thm:elim}.

The strategy is shown via an example on an operational region with high number of bins in Figure \ref{fig:Second_Strategy}. On the left side of the Figure \ref{fig:Second_Strategy}, suppose that the red dots are the path to be revealed by the Markov chain of the red swarm to converge to the base location, the shaded line represents the boundary bins determined by the blue swarm. As will be shown in the Lemma \ref{lem:mrs}, we can estimate the distribution of the total density of red swarm that enters the boundary bins\footnote{The distribution of the total density of red swarm that enters the boundary bins may be slightly different than estimated since the stochasticity of the Markov chain of the red swarm.}. The distribution of the total density of red swarm that enters the boundary bins is represented with a red distribution on the right side of Figure \ref{fig:Second_Strategy}. According to elimination dynamics given in the Eq. (\ref{eq:xbre_1}), blue swarm can eliminate all agents of the red swarm as convergence to the distribution of the total density of red swarm that enters the boundary bins. As shown in the Theorem \ref{thm:elim}, the maximum number of red agents that   enter the base location without being eliminated can be estimated for any determined boundary bins. Hence, the distance between the base location and boundary bins can be maximized as keeping the number of red agents that enter the base location under a desired number.

\begin{figure}[H]
\centering
\includegraphics[width=0.30\textwidth]{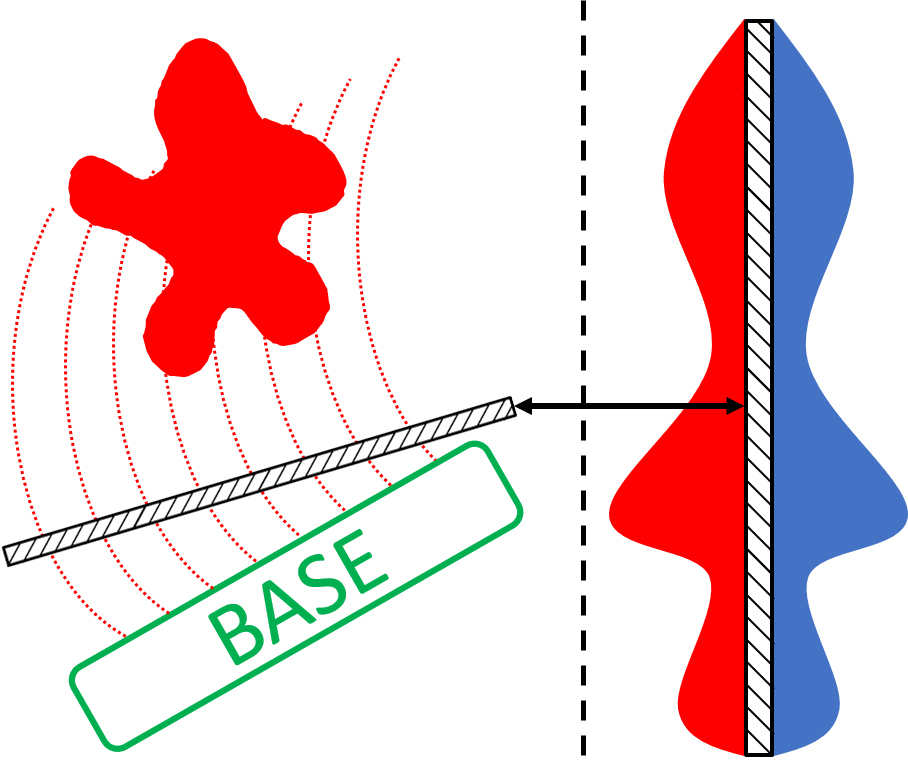}
\caption{Representation of the engagement strategy on an operational region with high number of bins. The red swarm converges to the base location . Red dots represent the path to be revealed by the policy of the red swarm. The shaded line represents the boundary bins determined by the blue swarm.}
\label{fig:Second_Strategy}
\end{figure}

Let $t_{fc}$ and $t_{lc}$ denote the first and last contact times that are the first and last time for any red agent to enter any determined boundary bin, respectively. This strategy is divided into two phase for before and after the time $t_{fc}$, respectively. 

\textbf{Phase 1} \textit{(Before the time $t_{fc}$):} In phase 1 distribution of the total density of red agents that enter the boundary bins, which is the desired distribution of the blue swarm, is estimated. In Theorem \ref{thm:elim}, it is proved that there exists a finite time $t_{fc}$ and there is not any elimination before the time $t_{fc}$ under some mild assumptions. Since there is no elimination until $t_{fc}$ and Markov chain of the red swarm does not change, the desired distribution of the blue swarm is a stationary distribution. Blue swarm designs a Markov chain to converge to this stationary distribution and uses the designed Markov chain until the time $t_{fc}$. 

\textbf{Phase 2} \textit{(After the time $t_{fc}$):} In phase 2, there are two different options that can be changed via convergence results and computation power of the blue swarm.

\textit{Option 1:} Markov matrix of the blue swarm turns to the identity matrix. Since the elimination changes the desired distribution of the blue swarm, propagation of the blue swarm with the same Markov chain may diverge to the blue swarm from the new desired distribution. 

\textit{Option 2:} If blue swarm cannot converge to the distribution of the total density of red agents that enter the boundary bins with a zero $L_1$ error, all agents of the red swarm cannot be eliminated. Hence, blue swarm may need to continue to converge between the times $t_{fc}$ and $t_{lc}$ to eliminate more red agents. Then, blue swarm can re-estimate the distribution of the total density of red agents that enter the boundary bins and design a new Markov chain after each elimination step. Memoryless property of the Markov-based algorithm allows blue swarm to assume each time-step as initial time-step. Re-designing of Markov chain of the blue swarm causes computation cost but it can provide better elimination results respect to turning the Markov matrix to identity matrix.

Estimation of the distribution of the total density of red agents that enter the boundary bins is represented on an example in Figure \ref{fig:M_r_s_Nodes}. Let make the following definitions for the algorithm. 

\begin{definition} \label{def:baseidxset}
\textit{(Index sets respect to the base location)}
The index set $I_b$ contains the bins belonging to the base location. The index set $I_{b-1}$ consists of the bins that are directly connected by adjacency matrix to the bins $i \in I_b$ and $I_{b-1} \cap I_b = \emptyset$, the index set $I_{b-2}$ consists of the bins that are directly connected by adjacency matrix to bins $i \in I_{b-1}$ and $I_{b-2} \cap (I_b \cup I_{b-1}) = \emptyset$, and so on.
\end{definition}

\begin{definition} \label{def:absorbing_states}
\textit{(Absorbing states)} In a Markov chain, if $M[i,i](k) = 1$ for a state $i \in R^{m \times m}$, the state $i$ is called as absorbing state for the time-step $k$.
\end{definition}

Bins of the operational region and Markov chain of the red swarm are modeled as vertices and transition probabilities of a graph in the example given in Figure \ref{fig:M_r_s_Nodes}. Assume that, red swarm is propagated to the base location that is the bin $R_N \in I_b$ for this case. They must enter the any bin from the set $I_{b-1} = \{R_{N-1}, R_{N-2} \}$, which are determined boundary bins, to reach to the base location. If we modify the states of the determined boundary to absorbing states as on the right side of the Figure \ref{fig:M_r_s_Nodes}, all agents of the red swarm will be stuck in the determined boundary bins. If we propagate red swarm with modified transition probabilities, we can estimate the distribution of the total density of red swarm that enters the boundary bins.

\begin{figure}[H]
\centering
\includegraphics[width=0.45\textwidth]{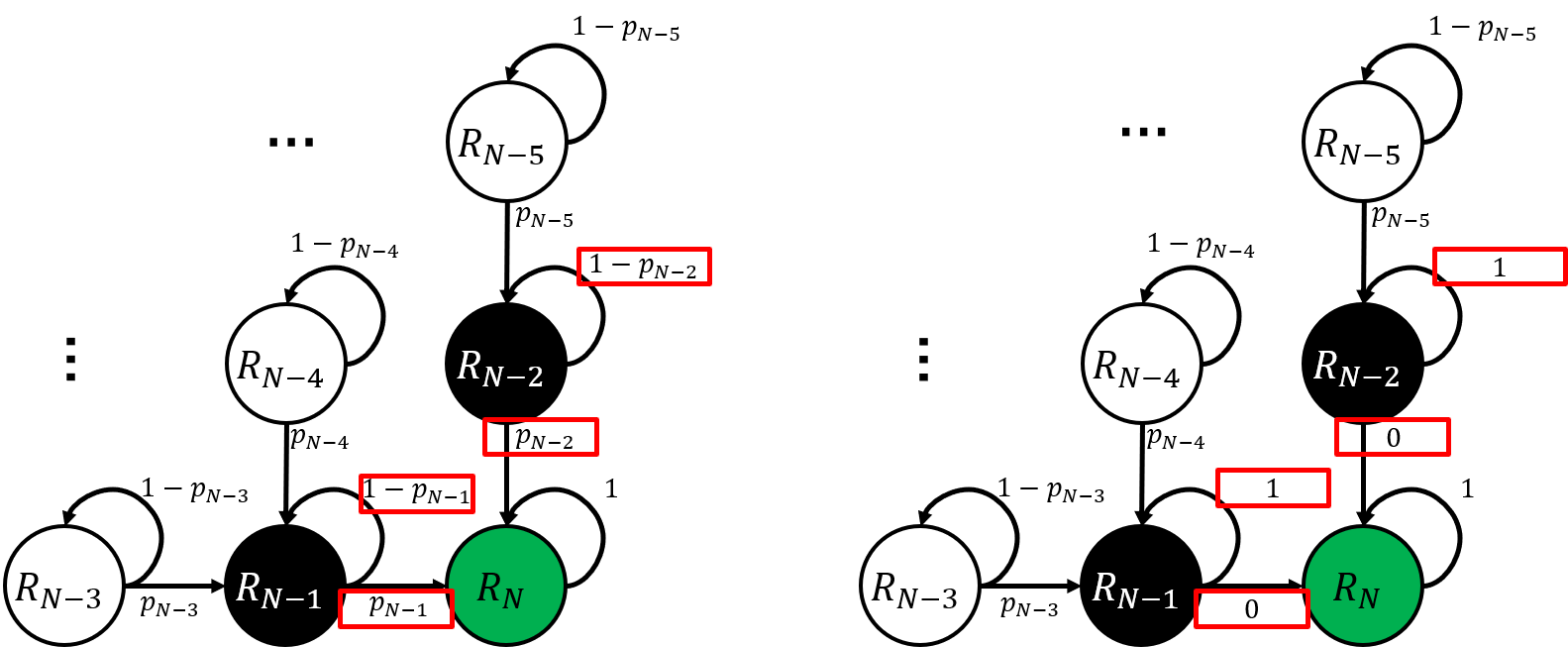}
\caption{The nodes of the graphs represent the bins of the operational region. Arrows represent the transition probabilities with directions for an agent (some possible transitions are ignored for the clarity of the figure). Green and black circles represent the bins that belong to the base location and determined boundary bins, respectively.}
\label{fig:M_r_s_Nodes}
\end{figure}

We will give the Lemma \ref{lem:mrs} for estimation of the distribution of the total density of red swarm in some determined boundary bins under some mild assumptions.


\begin{lemma} \label{lem:mrs}
    Assume that, red swarm is in the bins $i \in (I_{b-c} \cup I_{b-c-1} \cup ... \text{and so on.})$ for any $c = 1,2,...$ at time-step $k$ and satisfies the asymptotic convergence to the base location condition given in the Eq. (\ref{eq:br2rb}). 
    The Eq. (\ref{eq:srstfc}) gives the estimation of the distribution of total density of red agents that enter each bin $i \in I_{b-p}$ during propagation of the red swarm in the absence of any elimination in the bins $i \in (I_{b-p-1} \cup I_{b-p-2} \cup ... \text{and so on.})$ for any $p = 1,2,...,c$.
    \begin{equation} \label{eq:srstfc}
        \hat{x}_{r_s}(k+t_{lc}) = U_{r_s}(k:k+t_{lc}) x_r^+(k) \text{\quad where}
    \end{equation}
    \begin{equation} \label{eq:tlc}
        \begin{aligned}
            t_{lc} = \argmin_n \Big( U_{r_s}(k:k+n)x_{r}^+(k)[i] &= 0 \text{\; for \;} i \notin I_{b-p} \Big).\\
        \end{aligned}
    \end{equation}
    \begin{equation} \label{eq:urs}
        U_{r_s}(k:k+t_{lc}) =  M_{r_s}(k+t_{lc}-1)  M_{r_s}(k+t_{lc}-2) ...  M_{r_s}(k),
    \end{equation}
    \begin{equation} \label{eq:mrs}
        \begin{aligned}
        M_{r_s}[j,i](k+n) =&
            \begin{cases}
                M_{r}[j,i](k+n) & \text{if $i \not\in I_{b-p}$} \\
                1 & \text{if $i = j$ and $i \in I_{b-p}$} \\
                0 & \text{if $i \neq j$ and $i \in I_{b-p}$}
            \end{cases}\\
            & \text{ \quad for $n = 0,1,...,t_{lc}$,}
        \end{aligned}
    \end{equation}
\end{lemma}

\begin{proof}
    According to the Definition \ref{def:baseidxset}, since red swarm is in the bins $i \in (I_{b-c} \cup I_{b-c-1} \cup ... \text{and so on.})$ for any $c=1,2,...$ at time-step $k$ and satisfies the asymptotic convergence condition to the base location given in the Eq. (\ref{eq:br2rb}), there is a time that all agents of the red swarm enter the some bins of the set $I_{b-p}$ in the absence of any elimination in the bins $i \in (I_{b-p-1} \cup I_{b-p-2} \cup ... \text{and so on.})$ for any $p = 1,2,...,c$.

    Markov chain of the red swarm 
    can be modified as in the Eq. (\ref{eq:mrs}) to make the states of the determined boundary absorbing states. Hence, agents of the red swarm get stuck in the bins $i \in I_{b-p}$. 
    
    Considering that we propagate the red swarm with the modified Markov chain $\big($ $U_{r_s}$ $\big)$, there exist a time for the all agents of the red swarm to get stuck in the bins $i \in I_{b-p}$. Then, the last contact time can be calculated as in Eq. (\ref{eq:tlc}).

    Hence, we can estimate the distribution of the total density of red agents that enter any bin $i \in I_{b-p}$ as in Eq. (\ref{eq:srstfc}).

\end{proof}

\begin{assumption} \label{asm:motion_blue_red}
    \textit{(Capability of the blue swarm)} Blue and red swarms have the same transition constraints, so they have the same adjacency matrices and blue swarm is propagated from transient states to recurrent states using the shortest path. 
\end{assumption}

The Assumption \ref{asm:motion_blue_red} is given for the blue swarm to make Theorem \ref{thm:elim} suitable. A method is proposed in \cite{uzun2021-2} for the propagation from transient states to recurrent states using the both shortest path and density distribution on the recurrent states. We will give the following theorem for estimation of the maximum number of red agents that enter the base location for any determined boundary bins.


\begin{theorem} \label{thm:elim}

Assume that, red swarm is initialized as in the Lemma \ref{lem:mrs}, blue swarm have the properties given in the Assumption \ref{asm:motion_blue_red}, it is in the bins $i \in (I_b \cup I_{b-1} ... \cup I_{b-c+1})$ for any $c = 1,2,...$ at time-step $k$ and satisfies the asymptotic convergence condition $\lim_{n \to \infty} x_b(k+n) = \hat{x}_{r_s}(k+t_{lc})$. Let $N_{r_{ent}}$ denotes the total number of red agents that enter the any bin $i \in I_b$. The Eq. (\ref{eq:uptv}) gives the estimation of the maximum number of red agents that enter the any bin $i \in I_b$ respect to elimination dynamics given in the Eq. (\ref{eq:xbre_1}).
\begin{equation} \label{eq:ur_nrent}
   \hat{L}_r \geq N_{r_{ent}} \text{ where}
\end{equation}
\squeezeup
\begin{equation} \label{eq:uptv}
   \hat{L}_r = \bm{1}^T \max \bigg( \Big( \hat{s}_{r_s}(k+t_{lc}) - \hat{s}_{b}(k+t_{fc}) \Big), \bm{0} \bigg),
\end{equation}
\squeezeup
\begin{equation} \label{eq:srtlc}
    \hat{s}_{r_s}(k+t_{lc}) = N_r^+(k) \hat{x}_{r_s}(k+t_{lc}),
\end{equation}
\squeezeup
\begin{equation} \label{eq:sbtfc}
    \hat{s}_{b}(k+t_{fc}) = N_b^+(k) \hat{x}_{b}(k+t_{fc}),
\end{equation}
\squeezeup
\begin{equation} \label{eq:xbtfc}
    \hat{x}_{b}(k+t_{fc}) = U_b(k:k+t_{fc}) x_b^+(k),
\end{equation}
\squeezeup
\begin{equation} \label{eq:tfc}
    \begin{aligned}
        t_{fc} &= \argmin_n \Big(U_{r_s}(k:k+n) x_{r}^+(k)[i] \neq 0, \quad i \in I_{b-p} \Big)\\
    \end{aligned},
\end{equation}
for any $ p = 1, ...,c$.
\end{theorem}

\begin{proof}
    If blue swarm satisfies the asymptotic convergence condition $\lim_{n \to \infty} x_b(k+n) = \hat{x}_{r_s}(k+t_{lc})$, recurrent bins for the blue swarm is some bins from the set $I_{b-p}$ for $p=1,2,...,c$. 
    
    Since blue swarm is in the bins $i \in (I_b \cup I_{b-1} ... \cup I_{b-c+1})$, red swarm is in the bins $i \in (I_{b-c} \cup I_{b-c-1} \cup$ and so on $...)$ for any $c = 1,2,...$, blue swarm is propagated to its recurrent bins $i \in I_{b-p}$ for $p=1,2,...,c$ using shortest path and transition constraints are the same for the blue and red swarms, any agent from the red swarm cannot be eliminated in the bins $i \in (I_{b-p-1} \cup I_{b-p-2} \cup ...)$ which satisfies the absence of any elimination requirement in the Lemma \ref{lem:mrs}. Thus, we can estimate the distribution of total density and the total number of red agents that enter each bin $i \in I_{b-p}$ during propagation of the red swarm with the Eq. (\ref{eq:srstfc} and \ref{eq:srtlc}).
    
    Before the first contact time $t_{fc}$ given in the Eq. (\ref{eq:tfc}), we can find estimation of the distribution of the number of blue agents with the Eq. (\ref{eq:sbtfc}). According to the elimination dynamics given in the Eq. (\ref{eq:xbre_1}), the number of red agents that pass the bins $i \in I_{b-p}$ derived as $\hat{L}_r$ can be estimated with the Eq. (\ref{eq:uptv}) in cases where the blue swarm is propagated only until to time $t_{fc}$ and then uses identity matrix as policy.

    The estimation $\hat{L}_r$ is considered as the maximum number because blue swarm can re-design its Markov chain as updating the its desired distribution instead of using identity matrix after the first contact time $t_{fc}$. Also, red agents, which pass the bins $i \in I_{b-p}$ and propagate to the base location, can be eliminated by the blue agents that is in the bins $i \in (I_{b-p+1} \cup I_{b-p+2} \cup ... \cup I_{b*1}$) and have not yet converged to the desired distribution. Hence, the number of red agents that enter the bins $i \in I_{b}$ derived as $N_{r_{ent}}$ is satisfy the following equation,
    \begin{equation}
        \hat{L}_r \geq N_{r_{ent}}.
    \end{equation}
\end{proof}

\begin{corollary}
    If number of blue and red agents are equal to each other $\Big( N_b^+(k) = N_r^+(k) \Big)$, then Eq. (\ref{eq:uptv}) of Theorem \ref{thm:elim} can be simplified as,
    \begin{equation} \label{eq:uptv_eqNum}
        \hat{L}_r = N_b^+(k) \left\| \hat{x}_{r_s}(k+t_{lc}) - \hat{x}_{b}(k+t_{fc}) \right\|_{TV}. 
    \end{equation}
\end{corollary}
\begin{proof}
    If $N_b^+(k) = N_r^+(k)$, Eq. (\ref{eq:uptv}) can be re-written as,
    \begin{align*}
        \hat{L}_r &= \bm{1}^T \max \bigg( \Big( \hat{s}_{r_s}(k+t_{lc}) - \hat{s}_{b}(k+t_{fc}) \Big), \bm{0} \bigg)\\
            &= \bm{1}^T \max \bigg( \Big( N_r^+(k) \hat{x}_{r_s}(k+t_{lc}) - N_b^+(k) \hat{x}_{b}(k+t_{fc}) \Big), \bm{0} \bigg)\\
            &= N_b^+(k) \bm{1}^T \max \bigg( \Big( \hat{x}_{r_s}(k+t_{lc}) - \hat{x}_{b}(t_{k+fc}) \Big), \bm{0} \bigg).\\
    \end{align*}
    Since $\hat{x}_{r_s}(k+t_{lc})$ and $\hat{x}_{b}(k+t_{fc})$ are probability distributions, we can also simplify the equation as follows,
    \begin{align*}
        \hat{L}_r &= N_b^+(k) \bm{1}^T \max \bigg( \Big( \hat{x}_{r_s}(k+t_{lc}) - \hat{x}_{b}(k+t_{fc}) \Big), \bm{0} \bigg)\\
            &= N_b^+(k) \sum_{i \in \mathcal{R}} \max \bigg( \Big( \hat{x}_{r_s}(k+t_{lc})[i] - \hat{x}_{b}(k+t_{fc})[i] \Big), \bm{0} \bigg)\\
            &= N_b^+(k) \sum_{\substack{i \in \mathcal{R} \\ \hat{x}_{r_s}(k+t_{lc})[i] > \\ \hat{x}_{b}(k+t_{fc})[i]}}  \Big( \hat{x}_{r_s}(k+t_{lc})[i] - \hat{x}_{b}(k+t_{fc})[i] \Big).\\
    \end{align*}
    As shown in the Remark 4.3 of \cite{levin2017markov}, derived equation is equal to total variance distance between two probability distributions,
    \begin{align*}
        \hat{L}_r 
            &= N_b^+(k) \left\| \hat{x}_{r_s}(k+t_{lc}) - \hat{x}_{b}(k+t_{fc}) \right\|_{TV}.
    \end{align*}
\end{proof}

\begin{algorithm}[!hbt]
    \begin{algorithmic}[1]
    \State Initialize the error value $\epsilon_{opt} \in (0,1)$
    \ForEach{$p  \in \{1,2, ...,c \}$}
        \State $U_{r_s} \xleftarrow{}$ Obtain the $U_{r_s}$ for $I_{b-p}$ as in the Eq. (\ref{eq:urs})
        \State $t_{lc}, t_{fc} \xleftarrow{}$ Compute the $t_{lc}$ and $t_{fc}$ as in Eq. (\ref{eq:tlc} and \ref{eq:tfc})
        \State $\hat{x}_{r_s}(t_{lc}) \xleftarrow{}$ Calculate $\hat{x}_{r_s}(t_{lc})$ as in Eq. (\ref{eq:srstfc})
        \State $U_b \xleftarrow{}$ Synthesize $U_b$ for the distribution $\hat{x}_{r_s}(t_{lc})$
        \State $\hat{L}_r \xleftarrow{}$ Calculate the $\hat{L}_r$ as in the Eq. (\ref{eq:uptv}) or (\ref{eq:uptv_eqNum})
        \If{$ \big( \hat{L}_r / N_r^+(k) \big) < \epsilon_{opt}$ \textbf{ or } $p == 1$}
            \State $U_{b_{chsn}} = U_b$
        \Else 
            \State \textbf{return} $M_{b_{chsn}}$
        \EndIf
    \EndFor
    \end{algorithmic}
\caption{Optimal boundary bins for the engagement of swarms}
\label{alg:maxdist}
\end{algorithm}

The purpose of the blue swarm is to maximize the distance between the base location and the determined boundary bins by keeping the ratio of the estimation of the maximum number of red agents that enter the base location ($\hat{L}_r$) to the total number of red agents ($N_r$) lower than chosen $\epsilon_{opt} \in (0,1)$ value. The optimization problem can be solved with the given algorithm in Algorithm \ref{alg:maxdist}. 

In the first iteration of the algorithm, the first boundary that is directly connected by the adjacency matrix to the base location ($I_{b-1}$) is determined. Markov chain of the red swarm is modified with respect to the determined boundary bins and the first and last contact times of the red swarm to the boundary bins are calculated (lines 3-4). Projection of the distribution of red swarm on determined boundary bins is calculated (line 5). Then, Markov chain of the blue swarm is designed with respect to the projection of the distribution of red swarm on determined boundary bins and the maximum number of red agents that enter the bins $i \in I_{b-p}$ is estimated (lines 6-7). Finally, the designed Markov chain of the blue swarm is chosen (lines 8-9). For the following iterations, the same operations continue with increasing the distance between the base location and determined boundary bins but designed Markov chain of the blue swarm is chosen only if the ratio of the estimation of the maximum number of red agents that enter the base location ($\hat{L}_r$) to the total number of red agents ($N_r$) lower than chosen $\epsilon_{opt} \in (0,1)$ value. In case the blue swarm does not converge even to the first boundary, the first boundary must be chosen.

\section{Numerical Results}\label{sec:results}

In this paper, we consider four different scenarios of the swarm-to-swarm engagement problem. For all cases, red swarm synthesizes a Markov matrix using the proposed method given in the Section \ref{sec:somm} to converge to the stationary distribution of the base location. A desired distribution is determined using strategies given in Section \ref{sec:doct} for the blue swarm. Then, blue swarm synthesizes a Markov matrix using the proposed method given in the Section \ref{sec:somm} to converge to its determined desired distribution.

\begin{table*}[!hbt] 
\caption{Numerical results of the convergence to the projection of the red swarm on a determined boundary bins for several simulation cases.}
\centering
\begin{tabular}{@{}|c|c|c|@{}} 
\hline
\begin{tabular}[c]{@{}c@{}}Operational Region\end{tabular} &2D &3D\\ \hline
Desired total variation&
\begin{tabular}[c]{@{}c@{}}0.1\end{tabular} & 
\begin{tabular}[c]{@{}c@{}}0.1\end{tabular} \\ \hline
Estimated total variation w.r.t. determined boundary bins&
\begin{tabular}[c]{@{}c@{}}0.078\end{tabular} & 
\begin{tabular}[c]{@{}c@{}}0.066\end{tabular} \\ \hline
Number of agents in the swarms &
\begin{tabular}[c]{@{}c@{}}1000\end{tabular} & 
\begin{tabular}[c]{@{}c@{}}200\end{tabular} \\ \hline
Estimated number of agents that will enter the base location &
\begin{tabular}[c]{@{}c@{}}78\end{tabular} & 
\begin{tabular}[c]{@{}c@{}}13\end{tabular} \\ \hline
Number of agents that entered to the base location (Update : False) &
\begin{tabular}[c]{@{}c@{}}69\end{tabular} & 
\begin{tabular}[c]{@{}c@{}}21\end{tabular} \\ \hline
Number of agents that entered to the base location (Update : True) &
\begin{tabular}[c]{@{}c@{}}13\end{tabular} & 
\begin{tabular}[c]{@{}c@{}}7\end{tabular} \\ \hline
\end{tabular}
\label{table:2d_3d_false_true}
\end{table*}

Simulation results are shown in the Figures \ref{fig:2d_false} - \ref{fig:3d_true}. For all figures, current time-step, remaining time for the first contact to determined boundary bins for the red swarm, the estimated number of red agents that enter the base location, and information about the elimination for current time step are given in the upper-left corner of the figures. The number of remaining blue agents, red agents, and red agents that enter the base location is given in the middle top of the figures. Information about the colors on the figures is given in the upper-right corner of the figures. Numeric results belong to these scenarios are given in the Table \ref{table:2d_3d_false_true}.

\begin{figure*}[!hbt]
\begin{minipage}[t]{0.24\textwidth}
\centering
\includegraphics[width=0.98\textwidth]{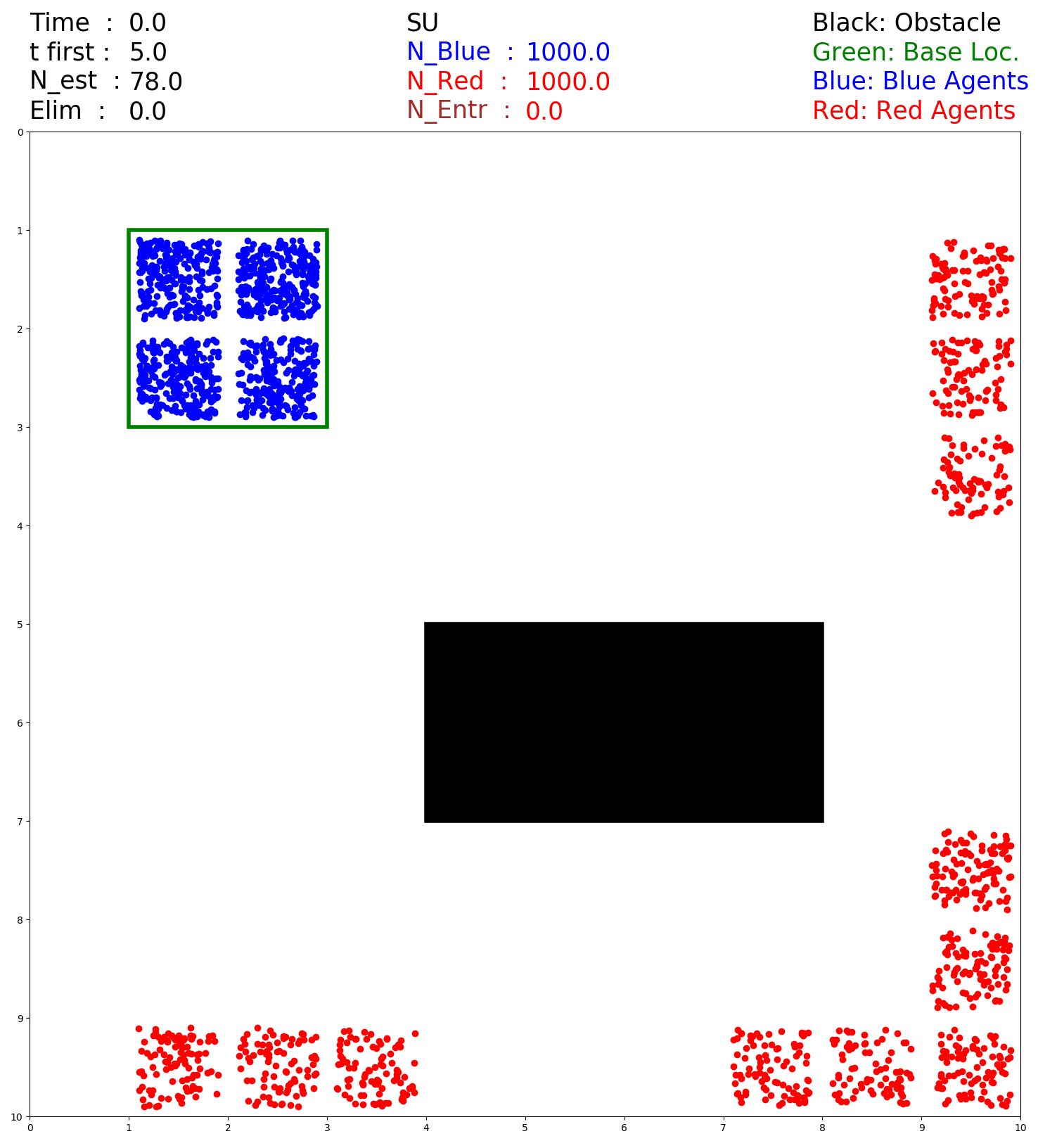}
\end{minipage}
\begin{minipage}[t]{0.24\textwidth}
\centering
\includegraphics[width=0.98\textwidth]{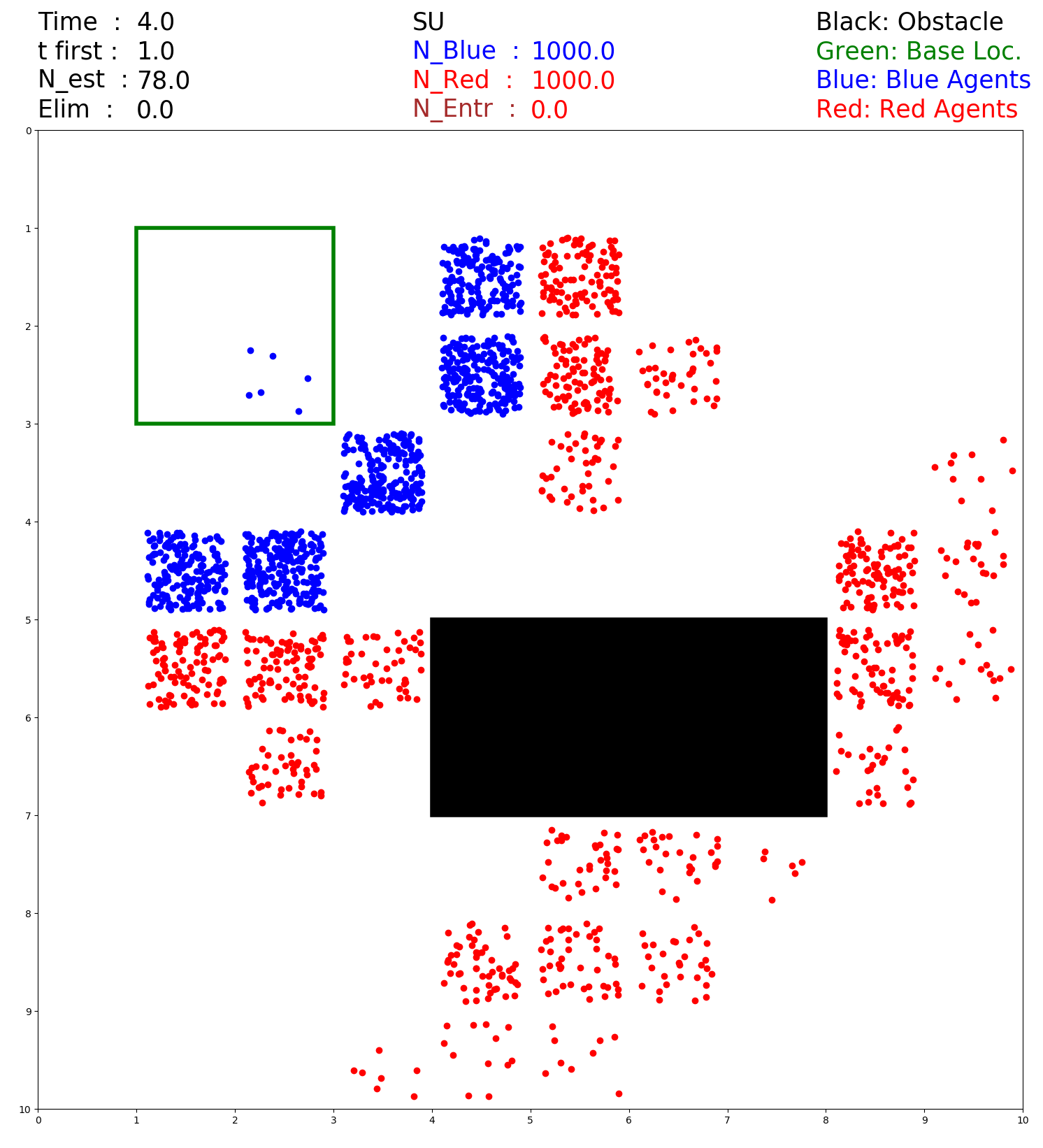}
\end{minipage}
\begin{minipage}[t]{0.24\textwidth}
\centering
\includegraphics[width=0.98\textwidth]{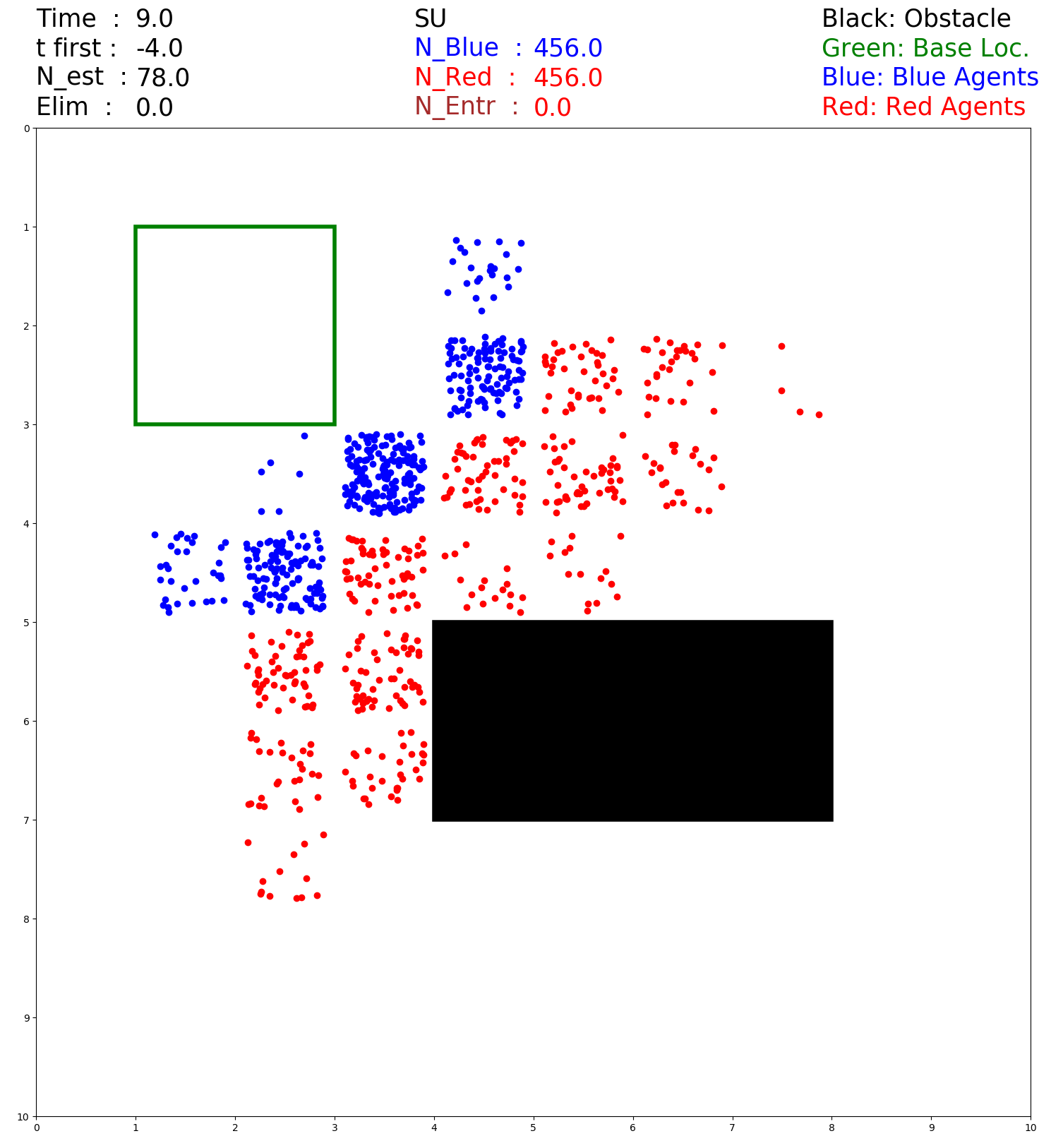}
\end{minipage}
\begin{minipage}[t]{0.24\textwidth}
\centering
\includegraphics[width=0.98\textwidth]{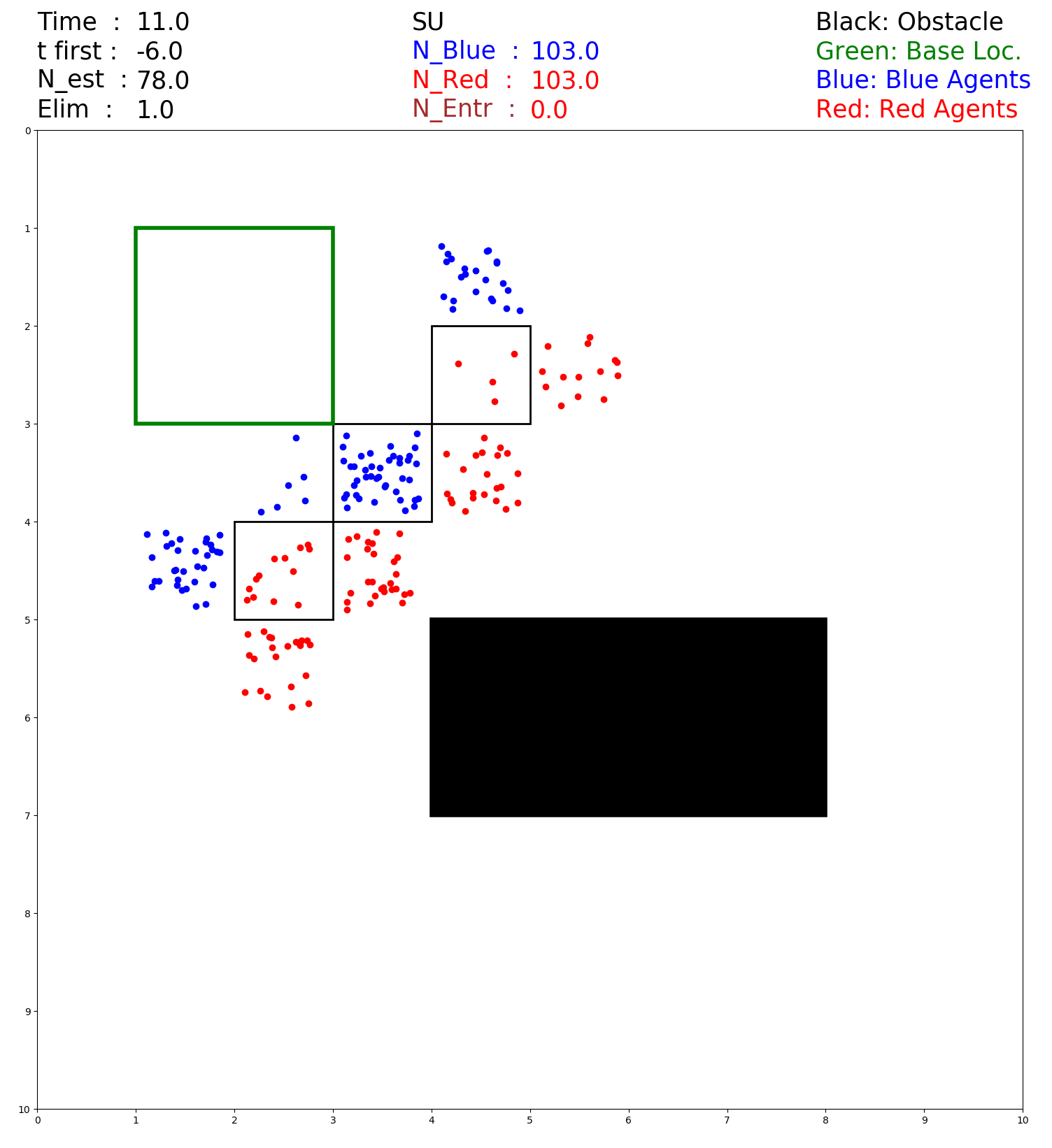}
\end{minipage}
\caption{Representation of the blue and red swarm on a 2D grid-map for several time instances when blue swarm performs boundary bins strategy given in Section \ref{sec:doct} to eliminate red swarm. Blue swarm did not update its desired distribution after each elimination step. This choice is introduced as Option 1 in Section \ref{sec:doct}.}
\label{fig:2d_false}
\end{figure*}

\begin{figure*}[!hbt]
\begin{minipage}[t]{0.24\textwidth}
\centering
\includegraphics[width=0.98\textwidth]{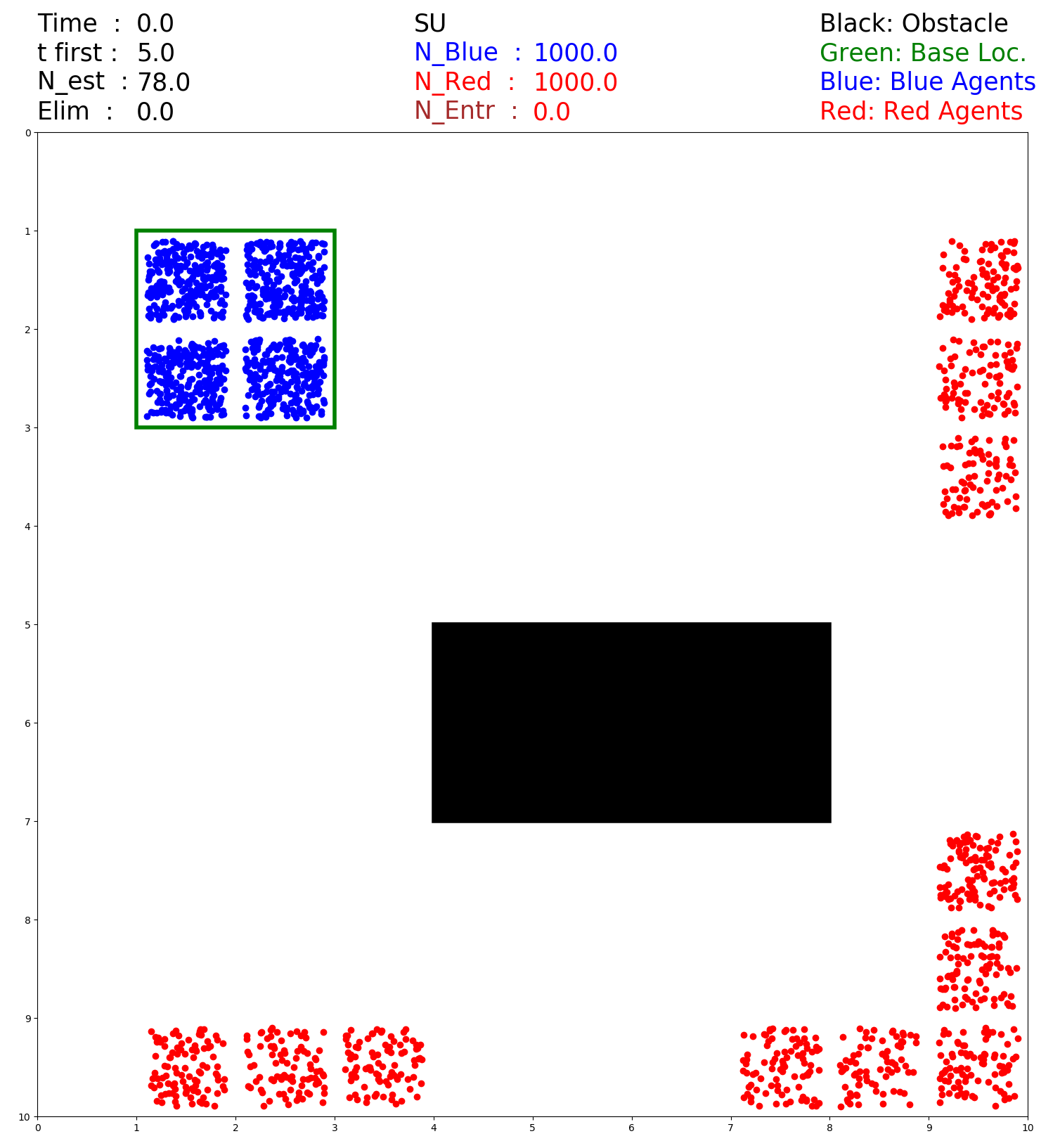}
\end{minipage}
\begin{minipage}[t]{0.24\textwidth}
\centering
\includegraphics[width=0.98\textwidth]{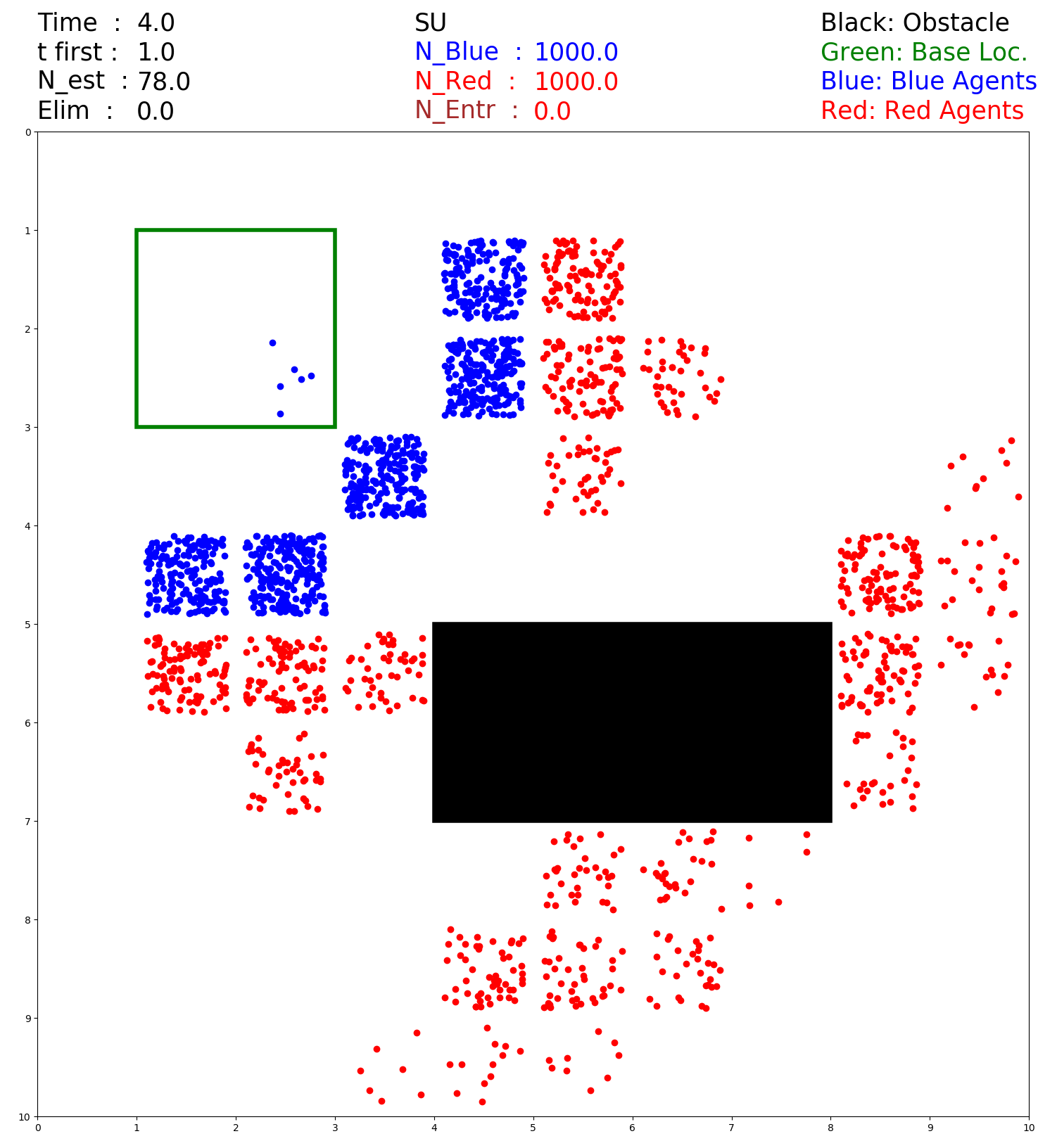}
\end{minipage}
\begin{minipage}[t]{0.24\textwidth}
\centering
\includegraphics[width=0.98\textwidth]{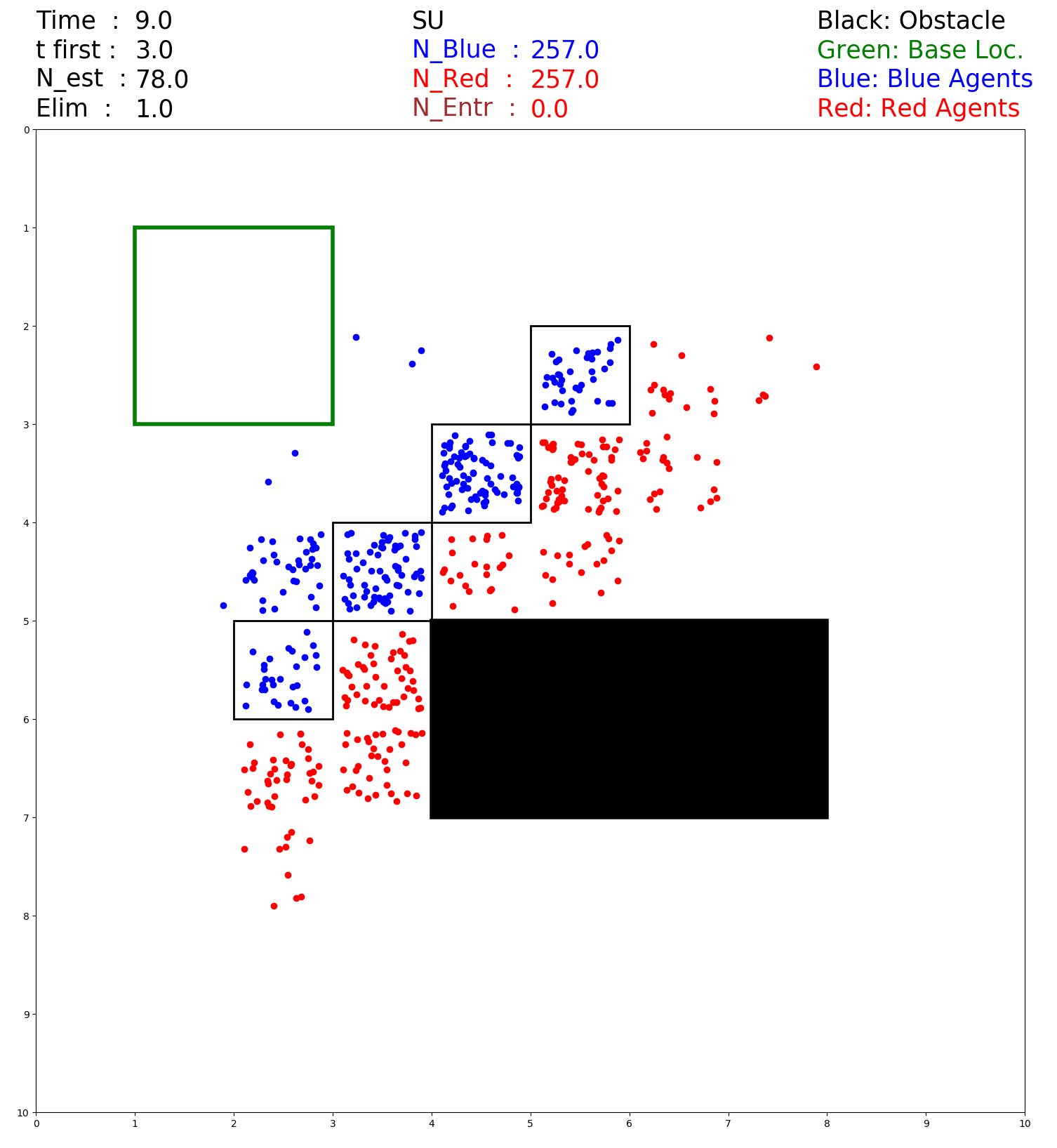}
\end{minipage}
\begin{minipage}[t]{0.24\textwidth}
\centering
\includegraphics[width=0.98\textwidth]{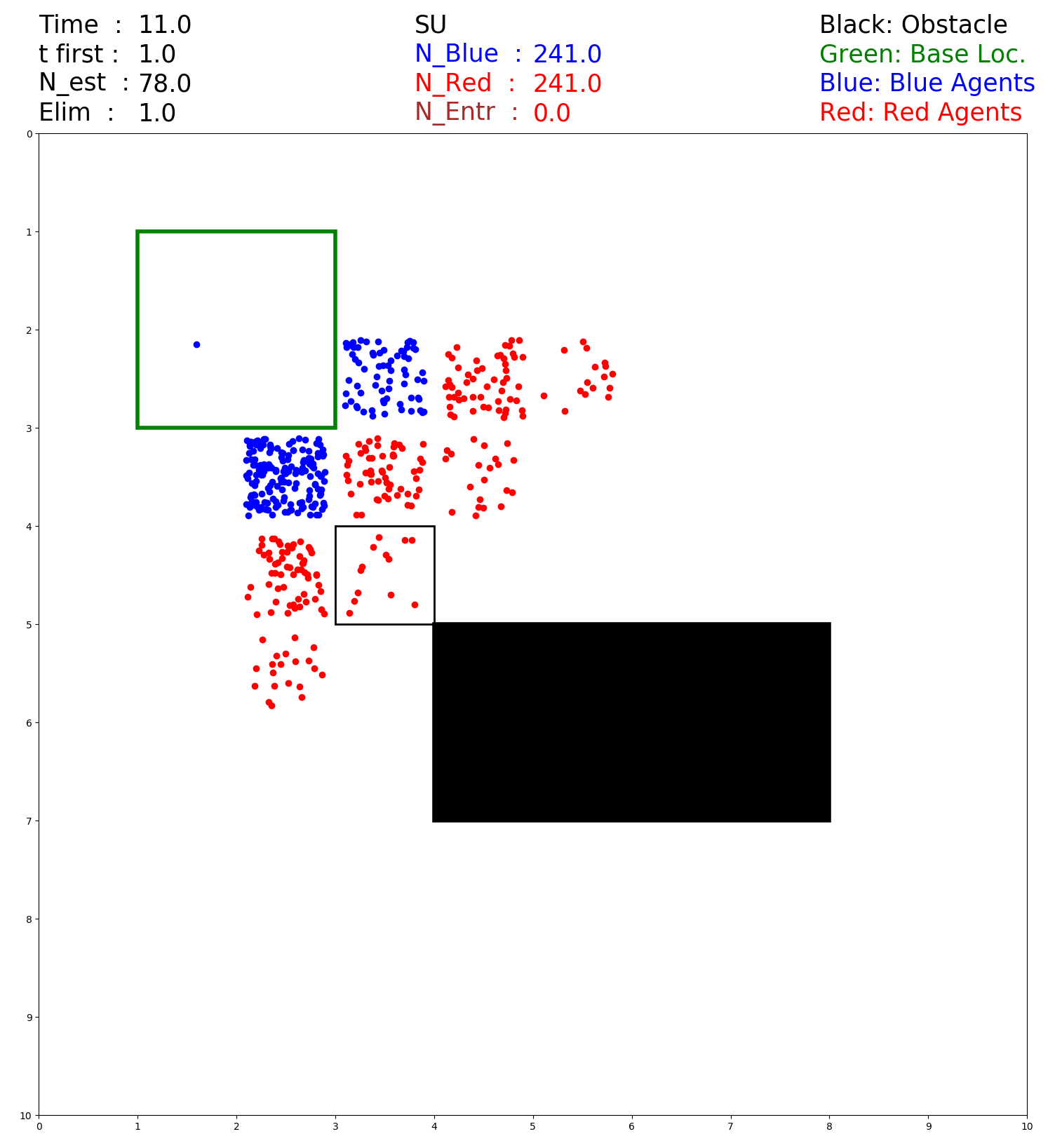}
\end{minipage}
\caption{Representation of the blue and red swarm on a 2D grid-map for several time instances when blue swarm perform boundary bins strategy given in Section \ref{sec:doct} to eliminate red swarm. Blue swarm updated its desired distribution after each elimination step. This choice is introduced as Option 2 in Section \ref{sec:doct}.}
\label{fig:2d_true}
\end{figure*}

In the first and second scenarios, a 2D grid-map, which consists of $8 \times 8$ bins, is considered. The grid-map contains a $2 \times 2$ base location on one corner. A number of thousands of blue and red agents are initialized at the base location and other sides of the map, respectively. There is a $2 \times 4$ obstacle in the middle of the grid-map. By simply adjusting the adjacency matrices of the swarms, they are prevented from entering the bins with the obstacle. 
In the first scenario, boundary bins are determined at the beginning of the scenario and blue swarm convergence to its desired distribution until the first contact time. Markov matrix of the blue swarm turns to an identity matrix after the first contact time as introduced as Option 1 in Section \ref{sec:doct}. In the second scenario, boundary bins are re-determined after each elimination step and blue swarm convergence to its new desired distribution until the new first contact time. Markov matrix of the blue swarm re-synthesized after each elimination step as introduced as Option 2 in the Section \ref{sec:doct}. Simulation results are shown in the Figure \ref{fig:2d_false} and \ref{fig:2d_true} for the different options, respectively. As can be seen in Table \ref{table:2d_3d_false_true}, the desired total variation to converge projection of the distribution of the red swarm on determined boundary bins is determined as $0.1$. The total variation is estimated as $0.078$ after boundary bins are chosen. Since there are $1000$ blue and red agents, the number of red agents that will enter the base location is estimated as $78$ using the Eq. (\ref{eq:uptv_eqNum}). The number of red agents that enter the base location is $69$ according to simulation results when Option 1 is used that is given in Section \ref{sec:doct}. When the desired distribution and Markov matrix of the blue swarm is updated after each elimination step as Option 2, the number of red agents that enter the base location is $13$ since blue swarm continues to converge after the first contact time.

\begin{figure*}[!hbt]
\begin{minipage}[t]{0.24\textwidth}
\centering
\includegraphics[width=0.95\textwidth]{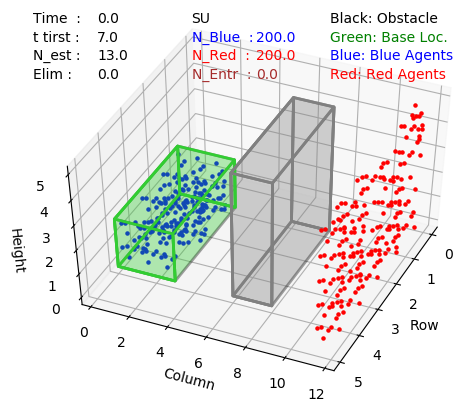}
\end{minipage}
\begin{minipage}[t]{0.24\textwidth}
\centering
\includegraphics[width=0.95\textwidth]{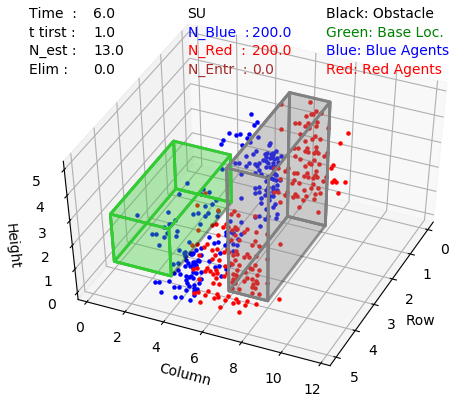}
\end{minipage}
\begin{minipage}[t]{0.24\textwidth}
\centering
\includegraphics[width=0.95\textwidth]{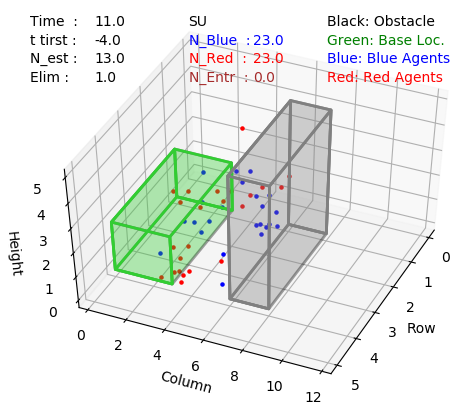}
\end{minipage}
\begin{minipage}[t]{0.24\textwidth}
\centering
\includegraphics[width=0.95\textwidth]{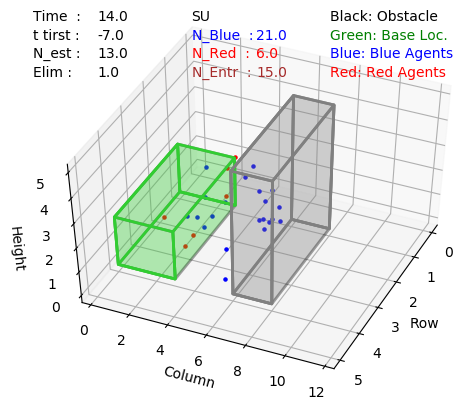}
\end{minipage}
\caption{Representation of the blue and red swarm on a 3D grid-map for several time instances when blue swarm performs boundary bins strategy given in Section \ref{sec:doct} to eliminate red swarm. Blue swarm did not update its desired distribution after each elimination step. This choice is introduced as Option 1 in Section \ref{sec:doct}.}
\label{fig:3d_false}
\end{figure*}

\begin{figure*}[!hbt]
\begin{minipage}[t]{0.24\textwidth}
\centering
\includegraphics[width=0.95\textwidth]{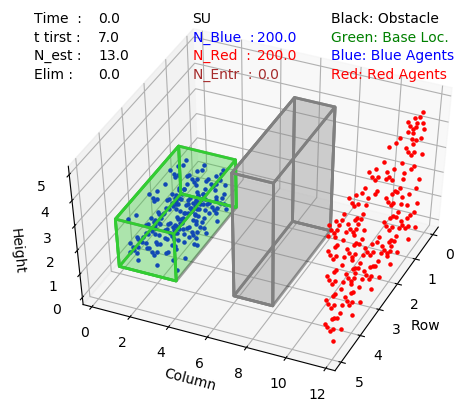}
\end{minipage}
\begin{minipage}[t]{0.24\textwidth}
\centering
\includegraphics[width=0.95\textwidth]{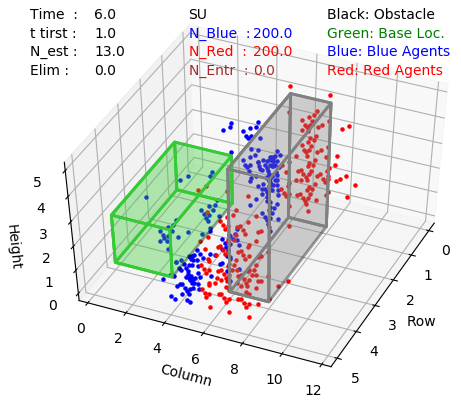}
\end{minipage}
\begin{minipage}[t]{0.24\textwidth}
\centering
\includegraphics[width=0.95\textwidth]{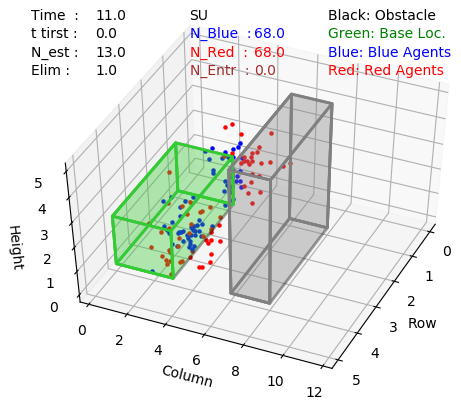}
\end{minipage}
\begin{minipage}[t]{0.24\textwidth}
\centering
\includegraphics[width=0.95\textwidth]{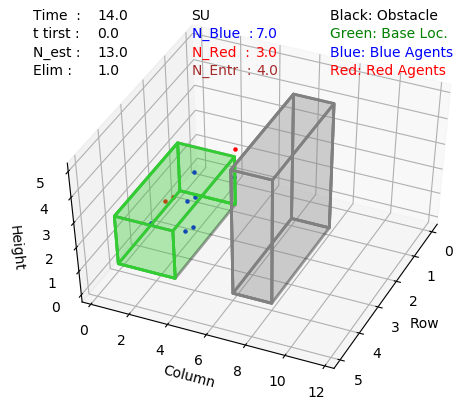}
\end{minipage}
\caption{Representation of the blue and red swarm on a 3D grid-map for several time instances when blue swarm performs boundary bins strategy given in Section \ref{sec:doct} to eliminate red swarm. Blue swarm updated its desired distribution after each elimination step. This choice is introduced as Option 2 in Section \ref{sec:doct}.}
\label{fig:3d_true}
\end{figure*}

In the third and fourth scenarios, other simulation results are given for a 3D grid-map, which consists of $5 \times 12  \times 5$ bins. The grid-map contains a $3 \times 3 \times 2$ base location on one side. A number of two hundreds of blue and red agents are initialized at the base location and opposite plane of the map, respectively. There is a $3 \times 2 \times 5$ obstacle in the middle of the grid-map. In the third scenario, Option 1 that is given in Section \ref{sec:doct} is used as the second scenario. In the fourth scenario, Option 2 is used as the third scenario. Simulation results are shown in the Figure \ref{fig:3d_false} and \ref{fig:3d_true} for the different options, respectively. As can be seen in Table \ref{table:2d_3d_false_true}, the desired total variation to converge projection of the distribution of the red swarm on determined boundary bins is determined as $0.1$. The total variation is estimated as $0.066$ after boundary bins are chosen. Since there are $200$ blue and red agents, the number of red agents that will enter the base location is estimated as $13$ using the Eq. (\ref{eq:uptv_eqNum}). The number of red agents that enter the base location is $21$ according to simulation results when Option 1 is used. This difference caused by the stochastic policy of the red swarm. Projection of the distribution of the red swarm is estimated using the Markov chain of the red swarm but the Markov chain is a stochastic policy and this situation causes a difference between the estimated and real distributions. When the desired distribution of the blue swarm is updated after each elimination step as Option 2, this difference dramatically decreases. The number of red agents that enter the base location is showed as $7$, since the reduced effect of stochasticity, and blue swarm continues to converge after the first contact time.

\section{Conclusion and Future Works}\label{sec:conclusion}
In this paper, swarm-to-swarm engagement problem is formulated and the dynamics of swarms are introduced. To the best of our knowledge, this work is the first approach to swarm-to-swarm engagement problem. When the purpose of the red swarm is to converge to a stationary desired distribution corresponds to a defended base location, the purpose of the blue swarm is to converge to the time-varying distribution and eliminate desired number of agents of the red swarm furthest from the base location. The strategy is based on convergence to the projection of the red swarm on any determined boundary bins. It is proved that the number of red agents that enter the defended base location can be bounded with respect to determination of the boundary bins.

There are lots of ideas about future works. Some of the are listed below.

\begin{itemize}

  \item It can be worked on optimal algorithms and their theoretical analysis that can converge to the time-varying distribution when properties such as elimination are included.
  
  \item Elimination process is assumed to be deterministic in this paper. Blue and red agents that are in the same bin eliminate each other in equal numbers. It can also be a random variable that varies with the ratio of the number of these blue and red agents.
  
   \item Learning methods can be used to estimate the policy of red swarm when the Markov matrix of the red swarm is unknown.

  \item Game theoretic methods can be used to make the blue swarm try to eliminate the red swarm as soon as possible while the red swarm tries to enter the base location without getting caught by the blue swarm.

  \item When agents belonging to the blue and red swarm enter the same bin, a game theoretic convexified model predictive control algorithm can be started for blue agents to catch red agents in the same continuous bin space.
  
\end{itemize}

\section{ACKNOWLEDGEMENTS}
The authors gratefully acknowledge Behçet Açıkmeşe from University of Washington for his valuable comments and feedback.

\bibliographystyle{IEEEtran}
\bibliography{main.bib}

\end{document}